\documentclass[a4paper,11pt]{amsart}

\usepackage{a4wide,amsmath,amsfonts,amssymb,mathrsfs,amsthm,bbm}
\usepackage[hyperfootnotes=false]{hyperref}

\allowdisplaybreaks[2]

\numberwithin{equation}{section}

\newtheorem{theorem}{Theorem}[section]
\newtheorem{lemma}[theorem]{Lemma}
\newtheorem{corollary}[theorem]{Corollary}
\newtheorem{remark}[theorem]{Remark}
\newtheorem{proposition}[theorem]{Proposition}
\newtheorem{definition}[theorem]{Definition}
\newtheorem{example}[theorem]{Example}
\newtheorem{assumption}[theorem]{Assumption}

\renewcommand{\d}{\mathrm{d}}
\newcommand{\dd}{\,\mathrm{d}}
\newcommand{\R}{\mathbb{R}}
\newcommand{\F}{\mathcal{F}}
\newcommand{\N}{\mathbb{N}}
\newcommand{\Z}{\mathbb{Z}}
\newcommand{\E}{\mathbb{E}}
\newcommand{\1}{\mathbf{1}}
\renewcommand{\P}{\mathbb{P}}

\title[One-dimensional game-theoretic differential equations]{One-dimensional game-theoretic differential equations}

\author[{\L}ochowski]{Rafa{\l} M. {\L}ochowski}
\address{Rafa{\l} M. {\L}ochowski, Warsaw School of Economics, Poland}
\email{rlocho314@gmail.com}

\author[Perkowski]{Nicolas Perkowski}
\address{Nicolas Perkowski, Free University of Berlin, Germany}
\email{perkowski@math.fu-berlin.de}

\author[Pr{\"o}mel]{David J. Pr{\"o}mel}
\address{David J. Pr{\"o}mel, University of Mannheim, Germany}
\email{proemel@uni-mannheim.de}

\date{\today}

\begin{document}
 
\begin{abstract}
  We provide a very brief introduction to typical paths and the corresponding It{\^o} type integration. Relying on this robust It{\^o} integration, we prove an existence and uniqueness result for one-dimensional differential equations driven by typical paths with non-Lipschitz continuous coefficients in the spirit of Yamada--Watanabe as well as an approximation result in the spirit of Doss--Sussmann. 
\end{abstract}

\maketitle




\section{Introduction}

In the past decade, ideas from game-theoretic probability (see the books \cite{Shafer2001,Shafer2019}) have led to various related notions of outer measures based on the concept of pathwise super-hedging coming from mathematical finance, see e.g. \cite{Vovk2009,Takeuchi2009,Vovk2012,Perkowski2016}. These outer measures allow to use arbitrage considerations to examine which path properties are satisfied by ``typical (price) paths'', that is, which path properties hold except null sets with respect to these outer measures. For instance, V. Vovk proved that non-constant typical continuous paths have infinite $p$-variation for $p<2$ and finite $p$-variation for $p>2$, see~\cite{Vovk2008}. Additionally, due to the financial nature of these outer measures, they found many applications in mathematical finance under model uncertainty, like robust versions of the pricing-hedging duality \cite{Beiglbock2017,Bartl2019,Bartl2020,Cheridito2019}, the role of measurability to avoid arbitrage~\cite{Vovk2017} and game-theoretic portfolio theory \cite[Chapter~17]{Shafer2019}.

Since typical paths are too irregular to apply classical calculus, a novel approach to work with typical paths is required. This motivated the development of a game-theoretic It{\^o} calculus, which can be viewed as a robust version of the stochastic It{\^o} calculus developed for (semi-)martingales like the Brownian motion, cf. for instance \cite{Karatzas1988,Revuz1999}. Game-theoretic ideas coming in particular from mathematical finance, indeed, allowed to set up a rather mature game-theoretic (or model-free) It{\^o} integration for typical paths, see \cite{Perkowski2016,Vovk2016,Lochowski2018}. Since the stochastic It{\^o} integration was originally motivated by stochastic differential equations, it rose the natural question whether the game-theoretic It{\^o} integration is sufficiently powerful to treat differential equations driven by typical paths. A first affirmative answer to this question was given by \cite{Bartl2019b}. Assuming the coefficients of the differential equations are Lipschitz continuous, \cite{Bartl2019b} introduces an outer measure allowing to show existence and uniqueness results for differential equations driven by typical paths even in a Hilbert space setting. Based on an alternative outer measure, the work \cite{Galane2018} provides also existence and uniqueness results for multi-dimensional differential equations with Lipschitz coefficients driven by typical paths. 

The general interest to gain a deeper understanding of differential equations driven by typical paths stems from their benefits on a theoretical as well as on an applied level. To illustrate these benefits, let us recall that differential equations perturbed by random noises are frequently used as mathematical models for real-world evolutions. As commonly observed, mathematical modeling comes with the issue of model uncertainty. One approach in  mathematical finance to treat such model uncertainty is to simultaneously work under a family of probability measures instead of under one fixed probability measure, cf. e.g.~\cite{Soner2011}. The solution theory for differential equations driven by typical paths allows immediately to simultaneously work under sensible families of probability measures, see Proposition~\ref{prop:properties of outer measure}, and such represents in this direction a more general approach to (stochastic) differential equations than the classical It{\^o} theory.  

In the present work we study one-dimensional differential equations driven by typical paths. Similarly to the theory of stochastic differential equations, we shall see that the one-dimensional case is special in certain aspects. For instance, it allows to treat stochastic differential equations with non-Lipschitz continuous diffusion coefficients, like the Cox--Ingersoll--Ross process, cf. Example~\ref{ex:CIR model}, which is a frequently used model for the evolution of interest rates and of the volatility on financial markets.

First, we show the existence of a unique solution to differential equations driven by typical paths with H{\"o}lder continuous diffusion coefficients of order~$1/2$, see Subsection~\ref{subsec:YW theorem}. This corresponds to the famous Yamada--Watanabe theorem~\cite{Yamada1971} from probability theory, which is only known to hold true in a basically one-dimensional setting. Like for the classical Yamada--Watanabe result, it is remarkable that differential equations driven by typical paths possess unique solutions under a non-Lipschitz assumption while usually uniqueness results for differential equations require the involved vector fields to be almost Lipschitz continuous. Intuitively, this demonstrates the regularizing effect of typical paths similar to the one of a Brownian motion. 
 
Secondly, we show a stability result between ordinary differential equations and differential equations driven by typical paths, see Subsection~\eqref{subsec:DS theorem}, saying that suitable ordinary differential equations can approximate a differential equation driven by typical paths. This result can be viewed as a game-theoretic version of the Doss--Sussmann theorem~\cite{Doss1977,Sussmann1978} from probability theory. The stability theorem of Doss and Sussmann can be again motivated by questions regarding the suitable mathematical model for a real-world evolutions. While for example stochastic differential equations driven by a Brownian motion form convenient mathematical models, the underlying real-world evolutions might be actually more suitably represented by a differential equation perturbed by a non-Markov noise process of bounded variation. Hence, a stability result like the Doss-Sussmann theorem is of upmost importance from a modeling perspective. 

\medskip
\noindent\textit{Organization of the paper:} Section~\ref{sec:game-theoretic probability} introduces the game-theoretic outer measure and the notion of typical paths. In Section~\ref{sec:Ito integration} we present the essential results regarding game-theoretic integration and provide an It{\^o} type formula. Section~\ref{sec:SDE} treats one-dimensional differential equations driven by typical paths.

\section{Game-theoretic probability and typical paths}\label{sec:game-theoretic probability}

We consider the sample space $\Omega:=C([0,\infty);\R)$, which is the space of all continuous functions $\omega \colon [0,\infty)\to \R$ with $\omega(0)=0$, and the coordinate process on $\Omega$ is denoted by $S=(S_t)_{t\in [0,\infty)}$ where $S_t(\omega):=\omega(t)$. We equip the space~$\Omega$ with the right-continuous filtration $(\mathcal{F}_t)_{t\in [0,\infty)}$ with $\mathcal{F}_t := \bigcup_{s>t} \sigma(S_u: u \le s)$ and set $\mathcal{F}:= \bigvee_{t\in [0,\infty)} \mathcal{F}_t$. Stopping times $\tau$ and the associated $\sigma$-algebras $\F_\tau$ are defined as usual. A mapping $F\colon \Omega \to \R$ is called \textit{game-theoretic variable} if $F$ is $\mathcal{F}$-measurable. The indicator function of a set $A$ is denoted by $\1_A$, $x\wedge y:=\min \{x,y\}$ for $x,y\in \R$ and the space $\R^d$ is equipped with the Euclidean norm~$\|\cdot\|$.
\medskip

A process $H \colon \Omega \times [0,\infty) \rightarrow \R$ is called a \emph{simple strategy} if it is of the form 
\begin{equation*}
  H_t(\omega) = \sum_{n= 0}^{\infty} F_n(\omega) \1_{(\tau_n(\omega),\tau_{n+1}(\omega)]}(t), \quad (\omega,t)\in \Omega\times [0,\infty),
\end{equation*}
where $F_n\colon \Omega \rightarrow \R$ are $\F_{\tau_n}$-measurable bounded functions for $n\in\mathbb{N}$ and $0 = \tau_0(\omega) \leq \tau_1(\omega) \leq \dots$ are stopping times such that for every $\omega \in \Omega$ one has $\lim_{n\to \infty} \tau_n(\omega) = \infty$ and for every interval $[s,t]\subset[0,\infty)$ there are at most finitely many stopping times $(\tau_n)_{n=N,\dots,M}$ satisfying $\tau_n\in [s,t]$. For such a simple strategy $H$ the corresponding capital process
\begin{equation*}
  (H \cdot S)_t(\omega) = \sum_{n=0}^\infty F_n(\omega) (S_{\tau_{n+1}(\omega) \wedge t}(\omega) - S_{\tau_n(\omega) \wedge t}(\omega))
\end{equation*}
is well-defined for every $\omega \in \Omega$ and every $t \in [0,\infty)$. A simple strategy $H$ is called $\lambda$-\textit{admissible} for $\lambda \geq 0$ if $(H\cdot S)_t(\omega) \ge - \lambda$ for all $t \in [0,\infty)$ and all $\omega \in \Omega$. 

The previous definitions come all with natural interpretation from a game-theoretic and a financial perspective. The sample space~$\Omega$ can be interpreted as the set of all possible price evolution on a financial market. In this context, a simple strategy~$H$ represents a trading strategy of an investor, who changes her position at preselected stopping times, and the process $((H \cdot S)_{t})_{t\in [0,\infty)}$ stands for the capital generated by trading according to $H$ into the price process $(S_t)_{t\in [0,\infty)}$. The admissible condition can be understood as a maximal credit limit as it is commonly imposed in mathematical finance. 

Like in classical mathematical finance, to consider trading only with respect to simple strategies is often not sufficient. Therefore, we need to work with all capital processes in the the liminf-closure of capital processes generated by simple strategies. This is comparable with the It\^o integral, which is an operator on the $L^2$-closure of simple integrands, except that we need to work in the present setting with a pointwise closure instead of a closure with respect to a probability measure. For this purpose, we introduce $\mathcal{H}_\lambda$ for the set of $\lambda$-admissible simple strategies and the set of capital gain processes by 
\begin{equation*}
  \mathcal{V}_\lambda :=\big\{ \mathcal{C}_\cdot=\liminf_{n\to \infty } (H^n\cdot S)_\cdot \,:\, (H^n)_{n\in \N} \subset \mathcal{H}_\lambda \big\}
\end{equation*}
for $\lambda\geq 0$. This allows us to introduce an outer measure and the notion of typical paths, as initiated by Vovk~\cite{Vovk2008}. While Vovk's original definition was based on a closure of simple strategies using countable convex combinations of them, we rely here on the set $\mathcal{V}_\lambda$. Of course, both ways to introduce an outer measure are naturally justified, see e.g. \cite[Section~2.3]{Perkowski2016}. The following definition presents the modified Vovk's outer measure as  introduced by \cite{Perkowski2015,Perkowski2016}.

\begin{definition}\label{def:Vovk's outer measure}
  Let $\tilde{\Omega}\subset \Omega$ be a non-empty set. The \textup{outer measure} $\overline{P}(\cdot;\tilde{\Omega})$ of the set~$A \subseteq \tilde{\Omega}$ is defined as the cheapest super-hedging price for $\1_A$, that is
  \begin{align*}
    \overline{P}(A;\tilde{\Omega}) &:= \inf\Big\{\lambda \geq 0: \exists \mathcal{C}\in \mathcal{V}_\lambda  \text{ s.t. } \forall \omega \in \tilde{\Omega} \text{ } \lambda + \liminf_{t\to \infty} \mathcal{C}_t(\omega)\ge \1_A(\omega)\, \Big\}.
  \end{align*}  
  A set of paths $A \subseteq \tilde{\Omega}$ is called a \textup{null set in} $\tilde{\Omega}$ if $\overline{P}(A;\tilde{\Omega})=0$. A property (P) holds for \textup{typical paths in} $\tilde{\Omega}$ if the set $A$ where (P) is violated is a null set w.r.t. $\overline{P}(\cdot;\tilde{\Omega})$. 
  
  Furthermore, we set $\overline{P}(A):=\overline{P}(A;\Omega)$ and say a property (P) holds for \textup{typical paths} if it holds for typical paths in $\Omega$.
\end{definition}

Keeping the financial interpretation of the above definitions in mind, the outer measure~$\overline{P}$ represents the cheapest super-hedging price, essentially as in the classical setting of mathematical finance but with one important difference: we require here super-hedging for all $\omega \in \Omega$ and not just almost surely. An additional reason leading to the great interest of the outer measure~$\overline{P}$ in mathematical finance under model uncertainty is that it dominates all local martingale measures on the space $\Omega$. Recall, a probability measure~$Q$ is called a (local) martingale measure if the coordinate process $(S_t)_{t\in [0,\infty)}$ is a (local) martingale w.r.t.~$Q$. Local martingale measures appear in mathematical finance to characterize arbitrage-free market models and as ``pricing'' measures for financial derivatives. 

Furthermore, a null set can essentially be viewed as a model-independent arbitrage opportunity of the first kind, cf. \cite[Lemma~3.2]{Perkowski2015}. Let us recall that, given a probability measure~$\P$ on $(\Omega, \F)$, we say that~ $(S_t)_{t\in [0,\infty)}$ satisfies \textit{no arbitrage of the first kind~(NA1)} under~$\P$ if the set 
\begin{equation*}
  \mathcal{W}^{\infty}_1 := \bigg\{ 1 + \int_0^\infty H_u \dd S_u\, : \, H \in \mathcal{H}_1\bigg \}
\end{equation*}
is bounded in probability, that is if $\lim_{n \to \infty} \sup_{X\in \mathcal{W}^{\infty}_1} \mathbb{P}( X \geq n)=0$. Here $\mathcal{H}_1$ stands for the set of all integrable processes w.r.t.~$\P$. 

\medskip

The following proposition collects properties of the outer measure~$\overline{P}$ and presents its link to mathematical finance, which we discussed vaguely in the previous paragraph. 

\begin{proposition}[Proposition~3.3 in \cite{Perkowski2015}]\label{prop:properties of outer measure}~

  \begin{enumerate} 
    \item[(i)] $\overline{P}$ is an outer measure with $\overline{P}(\Omega)=1$, i.e. $\overline{P}$ is non-decreasing, countably sub-additive, and $\overline{P}(\emptyset) = 0$.
    \item[(ii)] Let $\mathbb{P}$ be a probability measure on $(\Omega, \F)$ such that the coordinate process $(S_t)_{t\in [0,\infty)}$ is a local martingale under~$\mathbb{P}$, and let $A \in \F$. Then $\mathbb{P}(A) \le \overline{P}(A)$.
    \item[(iii)] Let $A \in \F$ be a null set, and let $\mathbb{P}$ be a probability measure on $(\Omega, \F)$ such that the coordinate process~$(S_t)_{t\in [0,\infty)}$ satisfies (NA1) under $\P$. Then $\mathbb{P}(A) = 0$.
  \end{enumerate}
\end{proposition}

\subsection{Dambis--Dubins--Schwarz theorem}

A very remarkable result in the context of game-theoretic probability is the Dambis--Dubins--Schwarz theorem due to Vovk~\cite{Vovk2012}. It connects the outer measure~$\overline{P}$ with the Wiener measure~$\mathbb{W}$ on $(\Omega,\mathcal{F})$, in a similar spirit like the classical Dambis--Dubins--Schwarz theorem connects (local) martingale measures with the Wiener measure relying on suitable time-changes, see e.g. \cite[Chapter~3, Theorem~4.6]{Karatzas1988}. In order to give the precise formulation of Vovk's game-theoretic Dambis--Dubins--Schwarz theorem, we recall the definition of time-superinvariant sets, cf.~\cite[Section~3]{Vovk2012}.

\begin{definition}\label{def:time-invariant}
  A continuous non-decreasing function $\phi \colon [0,\infty) \to [0,\infty)$ satisfying $\phi(0)=0$ is said to be a \textup{time-change}. The set of all time-changes will be denoted by $\mathcal{G}_0$ and the group of all time-changes that are strictly increasing and unbounded will be denoted by $\mathcal{G}$. A subset $A\subset \Omega$ is called \textup{time-superinvariant} if for each $\omega \in \Omega$ for all $\phi \in\mathcal{G}_0$ it holds that 
  \begin{equation}\label{eq:invariance}
    \omega \circ \phi \in A \quad \Rightarrow \quad \omega \in A.
  \end{equation}
  A set $A\subset \Omega$ is called \textup{time-invariant} if~\eqref{eq:invariance} holds true for all $\phi\in\mathcal{G}$.
\end{definition}

For a comprehensive and intuitive explanation of time-superinvariance we refer the interested reader to \cite[Remark~3.3]{Vovk2012}. With these definitions at hand, we can present Vovk's game-theoretic Dambis--Dubins--Schwarz theorem.
 
\begin{theorem}[Theorem~3.1 in \cite{Vovk2012}]\label{thm:DDS theorem}
  Each time-superinvariant set $A\subset \Omega$ satisfies $\overline{P}(A)=\mathbb{W}(A)$.
\end{theorem}

Of course, since we consider here a slight modification of Vovk's original outer measure, one needs to verify that the game-theoretic Dambis--Dubins--Schwarz theorem due to Vovk still holds for~$\overline{P}$, see \cite[Theorem~2.6]{Beiglbock2017}. Keeping in mind Theorem~\ref{thm:DDS theorem}, we observe the following regarding time-superinvariant sets: the definition of time-superinvariance ensures that all martingale measures on $\Omega$ assign a time-superinvariant set exactly the same probability. This also motivates the name ``Dambis--Dubins--Schwarz theorem'', which roughly says that every one-dimensional martingale is a time-changed Brownian motion.

\medskip

If one considers only sets of nowhere constant and divergent paths, the notions of time-superinvariance and time-invariance turn out to be equivalent. 

\begin{definition}\label{def:DS}
  A path $\omega\in \Omega$ is said to be \textup{nowhere constant} if there is no interval $(s,t)\subset [0,\infty)$ such that $\omega$ is constant on $(s,t)$ and $\omega$ is said to be \textup{divergent} if there is no $c \in \mathbb{R}$ such that $\lim_{t\to\infty} \omega(t)=c$. The set $\mathrm{DS} \subset \Omega$ denotes the set of all $\omega\in \Omega$ that are nowhere constant and divergent.
\end{definition}

\begin{lemma}[Lemma 3.5 in \cite{Vovk2012}]\label{lem:timeinvariant}
  A set $A \subset \mathrm{DS}$ is time-superinvariant if and only if it is time-invariant. 
\end{lemma}

\section{It{\^o} integration w.r.t. typical paths}\label{sec:Ito integration}

In the spirit of stochastic It{\^o} integration, it is possible to develop a game-theoretic integration theory for typical paths, see \cite{Perkowski2016,Vovk2016,Lochowski2018} and also \cite[Chapter~14]{Shafer2019}. Like the classical It{\^o} integration (\cite[Chapter~IV]{Revuz1999}), the game-theoretic integration requires as a fundamental ingredient the existence of quadratic variation for typical paths.

\subsection{Quadratic variation}

For a continuous path $\omega \colon [0,\infty) \to \R$ and $n\in \mathbb{N}$ we introduce the Lebesgue stopping times
\begin{align}\label{eq:Lebesgue stopping times}
  \sigma_0^n(\omega):=0 \quad \text{and}\quad
  \sigma_k^n(\omega):= \inf \big\{t\geq \sigma_{k-1}^n\,:\,\omega(t) \in 2^{-n}\Z \text{ and } \omega(t)\neq \omega (\sigma_{k-1}^n)\big\},
\end{align}
for $k\in \mathbb{N}$ and $\omega \in \Omega$. For $n\in \mathbb{N}$ the discrete quadratic variation of $\omega$ is given by
\begin{equation*}
  V^n_t(\omega) :=\sum_{k=0}^\infty \big(\omega({\sigma^n_{k+1}(\omega)\wedge t})-\omega({\sigma_{k}^n(\omega)\wedge t})\big)^2,\quad t\in [0,\infty).
\end{equation*}
To establish the convergence of the sequence~$(V^n_{\cdot}(\omega))_{n\in \N}$ of discrete quadratic variations, we recall the concept of locally uniform convergence in $C([0,\infty);\R)$. A sequence $(f_n)_{n\in \N}\subset C([0,\infty);\R)$ is said to \textit{converge locally uniformly} to $f\in C([0,\infty);\R)$ if 
\begin{equation*}
  \lim_{n\to \infty} \sup_{x\in [0,T]}\|f_n(x)-f(x)\|=0,\quad \text{for every }T>0. 
\end{equation*}
In this case, $f$ is called the \textit{locally uniform limit} of $(f_n)_{n\in \N}$.

\begin{proposition}\label{prop:quadratic variation}
   For typical paths $\omega\in \Omega$, the \textup{quadratic variation} 
  \begin{equation*}
    \langle S \rangle_t(\omega) := \lim_{n\to \infty} V^n_t(\omega),\quad t\in [0,\infty),
  \end{equation*}
  exists as a locally uniform limit in $C([0,\infty);\R)$. Moreover, for typical paths $\omega\in \Omega$, the quadratic variation $\langle S \rangle(\omega)\colon [0,\infty) \to \R$ is a non-negative and non-decreasing function. 
\end{proposition}

Proposition~\ref{prop:quadratic variation} can be found, for instance, in \cite[Lemma~8.1]{Vovk2012} and was generalized to typical paths in the space of c{\`a}dl{\`a}g functions with mildly restricted jumps (see \cite[Theorem~1]{Vovk2015}) and to typical paths in the space of c{\`a}dl{\`a}g functions satisfying a mild restriction on the jumps directed downwards (see \cite[Theorem~3.2]{Lochowski2018}).

\begin{remark}\label{rem:quadratic variation}
  The existence of quadratic variation $\langle S \rangle$ provided in Proposition~\ref{prop:quadratic variation} ensures the existence of quadratic variation in the sense of F{\"o}llmer, see \cite{Follmer1981}. Hence, typical paths can be used as integrators in the purely pathwise  It{\^o} calculus initiated by F{\"o}llmer~\cite{Follmer1981}.
\end{remark}

Similarly to the quadratic variation in probability theory, the existence of quadratic variation of a (typical) path is stable under time-changes. 

\begin{lemma}\label{lem:time-change quadratic variation}
  Let $\phi\colon [0,\infty)\to \R$ be a time-change in $\mathcal{G}_0$ and $\omega \in \Omega$. If the quadratic variation
  \begin{equation*}
    \langle S \rangle_t(\omega) := \lim_{n\to \infty} V^n_t(\omega),\quad t\in [0,\infty),
  \end{equation*}
  exists as a locally uniform limit in $C([0,\infty);\R)$, then 
  \begin{equation*}
    \langle S \rangle_t(\omega\circ \phi) = \lim_{n\to \infty} V^n_t(\omega\circ \phi),\quad t\in [0,\infty),
  \end{equation*}
  exists as a locally uniform limit in $C([0,\infty);\R)$ and 
  \begin{equation*}
    \langle S \rangle_{\phi(t)}(\omega) = \langle S \rangle_t(\omega\circ \phi),
    \quad t\in [0,\infty).
  \end{equation*}
\end{lemma}

\begin{proof}
  Keeping in mind the definition of the Lebesgue stopping times~$(\sigma^n_k)$ in~\eqref{eq:Lebesgue stopping times}, we notice that
  \begin{align*}    
    V^n_t(\omega\circ \phi)
    &= \sum_{k=0}^\infty \big(\omega\circ \phi({\sigma^n_{k+1}(\omega\circ \phi) \wedge t})-\omega\circ \phi({\sigma_{k}^n(\omega\circ \phi) \wedge t})\big)^2\\
    &= \sum_{k=0}^\infty \big(\omega({\sigma^n_{k+1}(\omega) \wedge\phi(t)})-\omega({\sigma_{k}^n(\omega) \wedge \phi(t)})\big)^2
    = V^n_{\phi(t)}(\omega),
  \end{align*}
  for all $t\in [0,\infty)$ and every $n\in \N$. Hence, $(V^n_{\cdot}(\omega))_{n\in \N}$ converges locally uniformly if and only if $(V^n_{\cdot}(\omega\circ \phi))_{n\in \N}$ converges locally uniformly, and for $t\in [0,\infty)$ we get
  \begin{align*}
    \langle S \rangle_t(\omega\circ \phi)
    =\lim_{n\to \infty} V^n_t(\omega\circ \phi)
    = \lim_{n\to \infty} V^n_{\phi(t)}(\omega)
    = \langle S \rangle_{\phi(t)}(\omega).
  \end{align*}
\end{proof}

\subsection{Game-theoretic integration}

Based on the existence of quadratic variation for typical paths, one can derive It{\^o}'s isometry type estimates w.r.t. the outer measure~$\overline{P}$, see \cite[Lemma~3.4]{Perkowski2016} or \cite[Lemma~4.5 or~4.8]{Lochowski2018}. While this allows to develop a comprehensive integration theory for typical paths (see \cite{Perkowski2016,Vovk2016,Lochowski2018} and \cite[Chapter~14]{Shafer2019}), we shall review here only the essential basics to treat differential equations driven by typical paths. To that end, we need to introduce some concepts concerning processes. 

\medskip

A process $X \colon \Omega \times [0,\infty) \to \mathbb{R}$ is called \textit{adapted} if the game-theoretic variable~$ X_{t}$ is $\mathcal{F}_{t}$-measurable for all $t \in [0,\infty)$. The process $X$ is said to be \textit{continuous} if the sample path $t \mapsto X_t(\omega)$ is continuous for typical paths~$\omega \in \Omega$. In order to work with ``game-theoretic'' processes, we recall the outer expectation~$\overline{E}$ associated to~$\overline{P}$. For a non-negative game-theoretic variable~$F$ we define 
\begin{align*}
  \overline{E}[F] := \inf\bigg\{\lambda \geq 0\,:\,  \exists \, \mathcal{C}\in \mathcal{V}_\lambda \text{ s.t. }\forall \omega \in \Omega  \text{ } \lambda + \liminf_{t\to \infty}\mathcal{C}_t(\omega) \ge F(\omega) \, \bigg\}.
\end{align*}
Given two processes $X,Y\colon \Omega \times [0,\infty)\to \R$ we introduce
\begin{equation*}
  \|X(\omega)-Y(\omega)\|_{\infty;[0,T]}:=\sup_{t\in [0,T]}\|X_t(\omega)-Y_t(\omega)\|,\quad \omega\in \Omega, 
\end{equation*}
for $T>0$. Now we identify two processes $X,Y$ if 
\begin{equation*}
  d_{T}(X,Y) := \overline E\big[\|X-Y\|_{\infty;[0,T]} \wedge 1\big]=0\quad \text{for all}\quad T\in [0,\infty).
\end{equation*}
The resulting space of equivalent classes of processes is denoted by $\overline L_{\text{loc}}^0([0,\infty);\R)$.

\medskip

In order to construct a game-theoretic It{\^o} integration, we start to define the integral of step functions and then extend the construction to a more general class of integrands. A process $F\colon \Omega \times [0,\infty) \rightarrow \R^d$ is called a \emph{step function} if there exist stopping times $0 = \tau_0 \leq \tau_1 \leq \dots$, and $\F_{\tau_n}$-measurable functions $F_n \colon \Omega \rightarrow \R^d$, such that for every $\omega \in \Omega$ we have $\tau_n(\omega) = \infty$ for all 
but finitely many $n$, and such that
\begin{equation*}
  F_t(\omega) = \sum_{n=0}^\infty F_n(\omega) \1_{[\tau_n(\omega),\tau_{n+1}(\omega))}(t),\quad (\omega,t)\in \Omega \times [0,\infty).
\end{equation*}
The corresponding integral of $F$ w.r.t. $(S_t)_{t\in [0,\infty)}$ is given by
\begin{equation*}
  (F\cdot S)_t := \sum_{n=0}^\infty F_{\tau_n} \big(S_{\tau_{n+1} \wedge t} - S_{\tau_n \wedge t}\big), \quad t \in [0,\infty),
\end{equation*}
which is well-defined for every $\omega \in \Omega$.

\medskip

The following lemma (Lemma~\ref{lem:integral}) and the next corollary (Corollary~\ref{cor:convergence of integrals}) are direct consequences of \cite[Theorem~3.5 and Corollary~3.6]{Perkowski2016}.

\begin{lemma}[Model-free It{\^o} integration]\label{lem:integral}
  Let $X\colon\Omega \times[0,\infty)\to \R$ be an adapted and continuous process. Then, there exists a process $\int X \dd S \in \overline L_{\text{loc}}^0([0,\infty);\R)$ with the following continuity property: for every $T>0$, if $(X^{(n)})_{n \in \N}$ is a sequence of simple functions and $(c_n)_{n\in \N}\subset\R$ is a sequence of real numbers such that $\lVert X^{(n)}(\omega) - X(\omega)\rVert_{\infty;[0,T]} \le c_n $ for all $\omega \in \Omega$ and all $n \in \N$, then for typical paths $\omega \in \Omega$ there exists a constant $C(\omega) > 0$ such that
  \begin{equation*}
    \Big\lVert (X^{(n)} \cdot S)(\omega) - \int X \dd S(\omega) \Big\rVert_{\infty;[0,T]} \le C(\omega) c_n \sqrt{\log n}
  \end{equation*}
  for all $n \in \mathbb{N}$.
  
  The integral process $\int X \dd S$ is continuous for typical paths, and there exists a representative $\int X \dd S$ which is adapted, although it may take the values $\pm \infty$. We usually write $\int_0^t X_s \dd S_s := \int X \dd S(t)$, and we call $\int X \dd S$ the \textup{model-free It{\^o} integral of $X$ w.r.t.~$S$}. 
\end{lemma}

One of fundamental properties of It{\^o} integrals in probability theory is that they can be approximated by left-point Riemann sums if one considers sufficiently regular integrands and a sufficiently strong concept of convergence. An example of this property is formulated in the next corollary for the model-free It{\^o} integral.

\begin{corollary}\label{cor:convergence of integrals}
  Suppose we are in the setting of Lemma~\ref{lem:integral}. If $c_n = o((\log n)^{-1/2})$, then for typical paths $((X^{(n)}_{\cdot} \cdot S))_{n\in \N}$ converges locally uniformly to $\int X \dd S$.
\end{corollary}

A more general integration theory for typical paths was developed in \cite{Perkowski2016}, \cite{Vovk2016} and \cite{Lochowski2018} providing, e.g., more sophisticated continuity estimates for the model-free It{\^o} integral and integration for not necessarily continuous integrands, and not necessarily continuous typical paths as integrators. 

\subsection{It{\^o}'s formula}

It is known that typical paths are as irregular as the sample paths of martingales. More precisely, typical paths have finite $p$-variation only for $p>2$, see \cite[Theorem~1]{Vovk2008}. Therefore, the model-free It{\^o} integral from Lemma~\ref{lem:integral} cannot satisfy the fundamental theorem of calculus but it does satisfy an It{\^o} type formula, as we shall show. 

\medskip

Let $A\colon \Omega \times [0,\infty)\to \R$ and $B\colon \Omega \times [0,\infty)\to \R$ be adapted and continuous processes. We consider the integral process $Y\colon \Omega \times [0,\infty)\to \R$ given by 
\begin{equation} \label{procY}
  Y_t :=  \int_0^t A_u \dd S_u +\int_0^t B_u \dd u,\quad t\in [0,\infty),
\end{equation}  
where the first integral denotes the model-free It{\^o} integral, as defined in Lemma~\ref{lem:integral}, and the second integral a classical Riemann--Stieltjes integral. For this type of integral processes we can derive the following It{\^o} type formula. 

\begin{proposition}\label{prop:Ito formula}
  If $(Y_t)_{t\in [0,\infty)}$ has the representation~\eqref{procY} and $f\colon\R \to \R$ is a twice continuously differentiable function, then the It{\^o} type formula 
  \begin{equation}\label{eq:Ito formula}
    f(Y_t) = f(Y_0)+ \int_0^t f^\prime (Y_u) B_u \dd u+ \int_0^t f^\prime (Y_u) A_u \dd S_u + \frac{1}{2} \int_0^t f^{\prime\prime}(Y_u) A_u^2 \dd \langle S\rangle_u, \quad 
  \end{equation}
  for $t\in [0,\infty)$, holds for typical paths.
\end{proposition}

Note, since $\langle S\rangle (\omega)$ exists and is a non-decreasing and continuous function for typical paths $\omega \in \Omega$, the integral $\int_0^t f^{\prime\prime}(Y_u) A_u^2 \dd \langle S\rangle_u$ in \eqref{eq:Ito formula} can be defined as a Riemann--Stieltjes integral. 

\begin{proof}[Proof of Proposition~\ref{prop:Ito formula}]
  Since $A\colon \Omega \times [0,\infty)\to \R$ and $B\colon \Omega \times [0,\infty)\to \R$ are adapted and continuous processes, we can locally uniformly approximate them by 
  \begin{equation*}
     A^{(n)}_t := \sum_{k=1}^\infty A_{\rho^n_{k-1}}\1_{[\rho^n_{k-1},\rho^n_{k})}(t)\quad\text{and}\quad
     B^{(n)}_t := \sum_{k=1}^\infty B_{\rho^n_{k-1}}\1_{[\rho^n_{k-1},\rho^n_{k})}(t),
  \end{equation*}
  for $n\in \N$, respectively, where 
  \begin{align*}
    \rho_0^n:=0 \quad \text{and}\quad
    \rho_k^n:= \inf \big\{t\geq \rho_{k-1}^n\,:\, |A_{t}-A_{\rho_{k-1}^{n}}|  \geq 2^{-n} \text{ or } |B_{t}-B_{\rho_{k-1}^{n}}|  \geq 2^{-n} \big\},
  \end{align*}
  for $k\in \N$. The corresponding approximation $(Y^{(n)}_{\cdot})_{n\in \N}$ of $(Y_t)_{t\in [0,\infty)}$ is defined by
  \begin{equation*}
    Y_t^{(n)} :=  \int_0^t A_u^{(n)} \dd S_u +\int_0^t B_u^{(n)} \dd u,\quad t\in [0,\infty).
  \end{equation*} 
  By the continuity of the model-free It{\^o} integration (use e.g. Corollary~\ref{cor:convergence of integrals} with $c_n:=2^{-n}$), we have
  \begin{equation*}
    \lim_{n\to \infty}\sup_{t\in [0,T]} \bigg\| \int_0^t A_u^{(n)} \dd S_u -\int_0^t A_u \dd S_u\bigg \|=0
  \end{equation*}
  for every $T>0$, for typical paths. Furthermore, by the continuity of Riemann--Stieltjes integration (see e.g. \cite[Proposition~2.7]{Friz2010}), we know that 
  \begin{equation*}
    \lim_{n\to \infty}\sup_{t\in [0,T]} \bigg\| \int_0^t B_u^{(n)}(\omega) \dd u -\int_0^t B_u(\omega) \dd u\bigg \|=0
  \end{equation*}
  for every $T>0$ and all $\omega\in \Omega$. Hence, $(Y^{(n)}_{\cdot})_{n\in\N}$ converges locally uniform to $(Y_t)_{t\in [0,\infty)}$ for typical paths. 
  
  We also notice that, for typical paths $\omega\in \Omega$, by the definition of quadratic variation, by the Cauchy--Schwarz inequality and Proposition~\ref{prop:quadratic variation}, the quadratic variation of $(Y^{(n)}_t)_{t\in [0,\infty)}$ exists and is given by 
  \begin{align*}
    \langle Y^{(n)} \rangle_t (\omega)
    &:= \lim_{n\to \infty }\sum_{k=0}^\infty \big(Y^{(n)}_{\sigma^n_{k+1}(\omega)\wedge t}(\omega)-Y^{(n)}_{\sigma_{k}^n(\omega)\wedge t}(\omega)\big)^2\\
    &= \lim_{n\to \infty }\sum_{k=0}^\infty \bigg(\int_0^{{\sigma^n_{k+1}\wedge t}} A_u^{(n)} \dd S_u (\omega) - \int_0^{{\sigma^n_{k}\wedge t}} A_u^{(n)} \dd S_u (\omega)\bigg)^2 \\
    &= \int_0^t (A_u^{(n)}(\omega))^2 \dd \langle S\rangle_u(\omega),\quad t\in [0,\infty),
  \end{align*}
  where the convergence takes place locally uniformly and $(\sigma^n_{k})_{k\in \N}$ denotes again the Lebesgue stopping times as defined in~\eqref{eq:Lebesgue stopping times}. Using Remark~\ref{rem:quadratic variation} and F{\"o}llmer's pathwise It{\^o} formula (\cite[TH{\'E}OR{\`E}ME]{Follmer1981}), we observe that 
  \begin{align}\label{eq:Ito discrete}
    \begin{split}
    &f(Y^{(n)}_t)- f(Y^{(n)}_0) \\
    &\quad= \int_0^t f^\prime (Y^{(n)}_u)\dd Y^{(n)}_u 
    + \frac{1}{2} \int_{0}^{t} f^{\prime\prime}(Y_u^{(n)}) \dd \langle Y^{(n)}\rangle_u\\
    &\quad= \int_0^t  f^\prime (Y_u^{(n)}) B^{(n)}_u \dd u 
    + \int_0^t f^\prime (Y_u^{(n)}) A^{(n)}_{u} \dd S_u 
    + \frac{1}{2} \int_{0}^{t} f^{\prime\prime}(Y^{(n)}_u) \big(A^{(n)}_{t}\big)^2 \dd \langle S\rangle_u.
    \end{split}
  \end{align}
  where we used the definition of $(Y^{(n)}_t)_{t\in [0,\infty)}$ and the previous identity for $(\langle Y^{(n)} \rangle_t)_{t\in [0,\infty)}$ in the last line. Since $(Y^{(n)}_{\cdot})_{n\in \N}$ converges locally uniformly to $(Y_t)_{t\in [0,\infty)}$ for typical paths, we can conclude the following as $n\to \infty$:
  \begin{itemize}
    \item $f(Y^{(n)}_\cdot)$, $f^\prime(Y^{(n)}_{\cdot})$, $f^{\prime\prime}(Y^{(n)}_{\cdot})$ converge locally uniformly to $f(Y_\cdot)$, $f^\prime(Y_{\cdot})$, $f^{\prime\prime}(Y_{\cdot})$, respectively, since $f$ is twice continuously differentiable;
    \item $\int_0^{\cdot}  f^\prime (Y_u^{(n)}) B^{(n)}_u \dd u$ and $\int_{0}^{\cdot} f^{\prime\prime}(Y^{(n)}_u) \big(A^{(n)}_{t}\big)^2 \dd \langle S\rangle_u$ converge locally uniformly to the limits $\int_0^{\cdot}  f^\prime (Y_u) B_u \dd u$ and $\int_{0}^{\cdot} f^{\prime\prime}(Y_u) \big(A_{t}\big)^2 \dd \langle S\rangle_u$, respectively, by the continuity of Riemann--Stieltjes integration;
    \item $\int_0^{\cdot} f^\prime (Y_u^{(n)}) A^{(n)}_{u} \dd S_u$ converges locally uniformly to $\int_0^{\cdot} f^\prime (Y_u) A_{u} \dd S_u$ by the continuity of the model-free It{\^o} integration.
  \end{itemize}
  Due to these observations about the convergence behavior, the identify~\eqref{eq:Ito discrete} reveals the assertion by sending $n\to \infty$.
\end{proof}

\section{Game-theoretic differential equations}\label{sec:SDE}

One of the main motivations to develop classical stochastic It{\^o} integration was to set up a well-posedness theory for stochastic differential equations. In a related manner, game-theoretic integration can be used to treat differential equations driven by typical paths, cf. \cite{Bartl2019} and \cite{Galane2018}.

\begin{remark}
  The work~\cite{Bartl2019b} provides existence and uniqueness results for differential equations on a finite time horizon~$[0,T]$ driven by typical paths in a Hilbert space setting (the typical paths attain their values in some Hilbert space), assuming that the coefficients are Lipschitz continuous. This approach relies on an extended path space. As a result one obtains a smaller outer measure than the outer measure defined in Definition~\ref{def:Vovk's outer measure}. By the extension of the path space we mean that together with the coordinate process $(S_t)_{t\in [0,T]}$ the investor is allowed to buy or sell assets, whose prices at the moment $t\in [0,T]$ are equal to $\Vert S\Vert^{2}_t-\langle S\rangle_t$, where here $\Vert \cdot \Vert$ denotes the Hilbert space norm and $(\langle S\rangle_t)_{t\in [0,T]}$ denotes the quadratic variation process of the coordinate process $(S_t)_{t\in [0,T]}$ but defined in a different way than the usual tensor quadratic variation of a Hilbert space-valued semi-martingale, see \cite[Remark~2.7]{Bartl2019b}. Additionally, the measure $\d \langle S\rangle$ is supposed to be absolutely continuous with respect to the Lebesgue measure $\d t$ and the density $\d \langle S\rangle/ \d t$ is supposed  to be globally bounded). 
 
  The work~\cite{Galane2018} obtains existence and uniqueness results for multi-dimensional differential equations driven by typical paths on a finite time horizon~$[0,T]$ under Lipschitz assumptions. In order to obtain a Burkholder--Davis--Gundy type inequality, \cite{Galane2018} is based on a modified outer expectation which may be interpreted as the super-hedging cost of not only the terminal value of some process~$(Z_t)_{t\in [0,T]}$, i.e. $Z_{T}$, but of the value $Z_{\tau}$ for any stopping time~$\tau$ such that $\tau\in[0,T]$. As a result one obtains a possibly greater outer measure than the outer measure defined in Definition~\ref{def:Vovk's outer measure}.
\end{remark}

While the case of multi-dimensional differential equations driven by typical paths was already studied, we focus here on \textit{one-dimensional} differential equations. The one-dimensional case is in various ways special and allows to obtain results which do not hold in general in a multi-dimensional setting. A famous example is the Yamada--Watanabe theorem providing the existence and uniqueness of a solution for differential equations with non-Lipschitz diffusion coefficients, see \cite{Yamada1971}. For examples and a more comprehensive discussion about the different necessary regularity assumptions on the coefficients in a one- and multi-dimensional setting, respectively, we refer to~\cite[Remark~2 and~3]{Watanabe1971}.

\medskip

In this section we consider one-dimensional differential equations driven by typical paths of the form
\begin{equation}\label{eq:SDE}
  X_t = x_0 + \int_0^t b(X_u)\dd \langle S\rangle_u + \int_0^t \sigma(X_u)\dd S_u,\quad t\in [0,\infty),
\end{equation}
where $x_0\in \R$, $b\colon\R\to \R $ and $\sigma\colon\R\to \R$ are continuous functions and $X\colon \Omega \times [0,\infty)\to\R$ is supposed to be an adapted and continuous process. Thanks to Lemma~\ref{lem:integral}, the model-free It{\^o} integral $\int_0^t \sigma(X_u)\dd S_u$ exists in $L^0_{\textup{loc}}([0,\infty);\R)$.

\begin{definition}
  Let $\tilde \Omega \subset \Omega$ be a set.
  \begin{enumerate}
    \item[(i)] We say that $X=(X_t)_{t\in [0,\infty)}$ is a \emph{solution} to \eqref{eq:SDE} in $\tilde \Omega$ if $X\colon \Omega\times [0,\infty)\to \time \to \R$ is an adapted and continuous process and \eqref{eq:SDE} holds for typical paths $\omega\in\tilde \Omega$. For $\Omega = \tilde \Omega$ we usually omit ``in $\Omega$'' and just call $(X_t)_{t\in [0,\infty)}$ a solution to~\eqref{eq:SDE}.
    \item[(ii)] We say that $X=(X_t)_{t\in [0,\infty)}$ is the \emph{unique solution} to \eqref{eq:SDE} in $\tilde{\Omega} $ if $(X_t)_{t\in [0,\infty)}$ is a solution to \eqref{eq:SDE} in $\tilde \Omega$ and, for every solution $Y=(Y_t)_{t\in [0,\infty)}$ in $\tilde \Omega$ we have $\|X(\omega)-Y(\omega)\|_{[0,T],\infty}=0$ for all $T\in [0,\infty)$, for typical paths $\omega\in \tilde{\Omega}$. For $\Omega = \tilde \Omega$ we usually omit ``in $\Omega$'' and just call $(X_t)_{t\in [0,\infty)}$ the unique solution to~\eqref{eq:SDE}.
  \end{enumerate}
\end{definition}

We shall study the differential equation~\eqref{eq:SDE} assuming the following regularity assumptions on the coefficients.

\begin{assumption}\label{ass:regularity}
  Suppose that $b\colon \R\to \R $ and $\sigma\colon \R\to \R$ are continuous and bounded functions satisfying the conditions 
  \begin{align*}
    |b(x)-b(y)|\leq C_b|x-y|\quad \text{and}\quad
    |\sigma (x)-\sigma(y)|\leq C_{\sigma}|x-y|^{1/2},
  \end{align*}
  for all $x,y\in \R$, where $C_b$ and $C_{\sigma}$ are positive constants.
\end{assumption}

\begin{example}\label{ex:CIR model}
  In mathematical finance the Cox--Ingersoll--Ross (CIR) process serves as a frequently applied model for the evolution of interest rates or volatility on financial market. The CIR process $(r_t)_{t\in [0,\infty)}$ can be described by the stochastic differential equation
  \begin{equation*}
    \d r_t = a (b-r_t)\dd t+ \sigma  \sqrt{r_t}\dd W_t,\quad t\in [0,\infty),
  \end{equation*}
  where $(W_t)_{t\in [0,\infty)}$ denotes a standard Wiener process and $a,b,\sigma$ are constants. 
  
  While the diffusion coefficient $x\mapsto \sigma \sqrt{x}$ is not Lipschitz continuous, it satisfies the regularity assumption of Assumption~\ref{ass:regularity}.
\end{example}

\subsection{Existence theorem}

To prove the existence of a solution to the differential equation~\eqref{eq:SDE} driven by typical paths, we introduce the following Euler type approximation: For $n\in \N$ we set $X^{(n)}_0=x_0$ and
\begin{equation*}
   X^{(n)}_t:= X_{\tau^n_k}^{(n)} + b(X_{\tau^n_k}^{(n)})(\langle S\rangle_t -\langle S \rangle_{\tau_k^n} ) + \sigma (X_{\tau^n_k}^{(n)})(S_{t}-S_{\tau^n_k}) 
\end{equation*}
for $t\in (\tau_k^n,\tau_{k+1}^n]$ where $\tau_0^n:=0$ and, for $k\in \N$,
\begin{align*}
  \tau_k^n(\omega):= \inf \big\{t\geq \tau_{k-1}^n\,:\, (S_t\in 2^{-n}\Z \text{ and } S_t\neq  S_{\tau_{k-1}^n}) \text{ or }(\langle S\rangle_t \in 2^{-n}\Z \text{ and } \langle S\rangle_t \neq \langle S\rangle_{\tau_{k-1}^n})\big\}.
\end{align*}

\begin{proposition}\label{prop:existence result}
  Suppose Assumption~\ref{ass:regularity} holds true. For typical paths $\omega \in \Omega$, the limit
  \begin{equation*}
     X_t:= \lim_{n\to\infty} X^{(n^3)}_t,\quad t\in [0,\infty), 
  \end{equation*}
  exists as locally uniform limit and $(X_t)_{t\in [0,\infty)}$ is a solution to the differential equation~\eqref{eq:SDE}.
\end{proposition}

As a preparation for the proof of Proposition~\ref{prop:existence result}, we show that the assertion holds true under the Wiener measure~$\mathbb{W}$ on $(\Omega,\mathcal{F})$.

\begin{lemma}\label{lem:approximation under Wiener measure}
  Suppose Assumption~\ref{ass:regularity} holds true. For all $T > 0$ there exists a constant~$C$ which only depends on $T, C_b, C_{\sigma}$ and, such that 
  \begin{equation}\label{eq:L1 convergence}
    \mathbb{E} \bigg[\sup_{t\in[0,T]} | X^{(n)}_t -X_t |\bigg] \leq \frac{C}{n^{1/2}},\quad n\in \N,
  \end{equation}
  where $\E$ denotes the expectation operator with respect to the Wiener measure~$\mathbb{W}$ on $(\Omega,\mathcal{F})$. In particular, $(X^{(n^3)}_{\cdot})_{n\in \N}$ converges almost surely locally uniformly to $(X_t)_{t\in [0,\infty)}$, where $X=(X_t)_{t\in[0,\infty)}$ denotes the strong solution to~\eqref{eq:SDE} under the Wiener measure $\mathbb{W}$, that is $(X_t)_{t\in [0,\infty)}$ is an adapted process on the filtered probability space $(\Omega,\mathcal{F},(\mathcal{F}_t)_{t\in [0,\infty)},\mathbb{W})$ such that $(X_t)_{t\in [0,\infty)}$ satisfies \eqref{eq:SDE} almost surely and $\mathbb{W} ( \int_0^t |b(X_u)|\dd \langle S\rangle_u + \int_0^t \sigma^2(X_u)\dd S_u <\infty)=1$ for every $t\in [0,\infty)$.
\end{lemma}

\begin{proof}
  We define
  \begin{equation*}
    \kappa_n (t):= \sum_{k = 0}^{\infty} \1_{(\tau^n_k, \tau^n_{k + 1}]} (t)\tau_k^n,\quad t\in [0,\infty).
  \end{equation*}
  Notice that
  \begin{equation*}
    X^{(n)}_t = x_0 + \int_0^t b (X^{(n)}_{\kappa_n (u)}) \dd \langle S \rangle_u + \int_0^t \sigma (X^{(n)}_{\kappa_n (u)}) \dd S_u, \quad t\in [0,\infty),
  \end{equation*}
  and that by the definition of $\tau^n_k$ we have
  \begin{equation*}
    |X^{(n)}_u - X^{(n)}_{\kappa_n (u)} | \leq C_{b,\sigma} 2^{- n},
  \end{equation*}
  where $C_{b,\sigma}\geq 1$ is a positive constant depending on the bounds of $b$ and $\sigma$. Moreover, recall that there exists a strong solution $(X_t)_{t\in[0,\infty)}$ to the stochastic differential equation 
  \begin{equation*}
    X_t = x_0 + \int_0^t b (X_u) \dd \langle S \rangle_u + \int_0^t \sigma (X_u) \dd S_u,\quad t\in [0,\infty),
  \end{equation*}
  under the Wiener measure~$\mathbb{W}$ by classical results from probability theory, see e.g. \cite[Section~5.2~C]{Karatzas1988}. For $n\in \N$ we introduce the process $(Y^{(n)}_t)_{t\in [0,\infty)}$ with $Y^{(n)}_t := X_t - X^{ (n)}_t$.

  \textit{Step~1:}
  Analogously to \cite[Proposition~2.2]{Gyongy2011} we claim: there exists $C > 0$, depending on $C_b, C_{\sigma}$ and $C_{b,\sigma}$ only, such that
  \begin{align}\label{eq:Y-n-bound}
    | Y^{(n)}_t | & \leq \frac{1 + C \langle S \rangle_t}{n} + C_b
    \int_0^t | Y^{(n)}_u | \dd \langle S \rangle_u + M^{(n)}_t,\quad t\in [0,\infty),
  \end{align}
  where $(M^{(n)}_t)_{t\in [0,\infty)}$ is some continuous local martingale with quadratic variation
  \begin{equation*}
    \langle M^{(n)} \rangle_t 
    \leq 2 C_{\sigma}^2C_{b,\sigma} \int_0^t (| Y^{(n)}_u | + 2^{- n}) \dd \langle S \rangle_u, 
    \quad t\in [0,\infty).
  \end{equation*}
  
  For $\delta >1$ and $\varepsilon >0$ let $\Psi_{\varepsilon, \delta}$ be the same function as in the proof of \cite[Proposition~2.2]{Gyongy2011}, i.e. $\Psi_{\varepsilon, \delta} (x) \leq \frac{2}{x \log (\delta)}$ and $\Psi_{\varepsilon, \delta}$ is non-negative and supported on $[\varepsilon / \delta, \varepsilon]$, and $\int_{\mathbb{R}} \Psi_{\varepsilon, \delta} (x) \dd x = 1$. Let $\Phi_{\varepsilon, \delta} (x) = \int_0^{| x |} \int_0^y \Psi_{\varepsilon, \delta} (z) \dd
  z \dd y$. Then, $\Phi_{\varepsilon, \delta}$ is twice continuously differentiable and we have $| x | \leq \varepsilon + \Phi_{\varepsilon, \delta} (x)$, $ |\Phi_{\varepsilon, \delta}' (x) | \leq 1$ and
  \begin{equation}\label{eq:Phi second derivative}
     \Phi_{\varepsilon, \delta}'' (x) = \Psi_{\varepsilon, \delta} (| x |)
     \leq \frac{2}{| x | \log (\delta)} \1_{[\varepsilon / \delta,
     \varepsilon]} (| x |)
     \leq \frac{2\delta}{\epsilon \log (\delta)}. 
  \end{equation}
  For $t\in [0,\infty)$, It{\^o}'s formula (cf. Proposition~\ref{prop:Ito formula}) yields
  \begin{align*}
   \Phi_{\varepsilon, \delta} (Y^{(n)}_t)
     &=  \Phi_{\varepsilon, \delta} (Y^{(n)}_0) +  M^{(n)}_t \\
     &\,\,\,+\int_0^t \bigg( \Phi_{\varepsilon, \delta}' (Y^{(n)}_u)
    (b (X_u) - b (X^{(n)}_{\kappa_n (u)})) + \frac{1}{2} \Phi_{\varepsilon,
    \delta}'' (Y^{(n)}_u) | \sigma (X_u) - \sigma (X^{(n)}_{\kappa_n (u)}) |^2
    \bigg) \dd \langle S \rangle_u 
  \end{align*}
  with 
  \begin{equation*}
    M^{(n)}_t:= \int_0^t \big( \Phi_{\varepsilon, \delta}' (Y^{(n)}_u) \big( \sigma (X_u) - \sigma (X^{(n)}_{\kappa_n (u)})\big)\big) \dd S_u.
  \end{equation*}  
  Hence, by the properties of $\Phi_{\varepsilon, \delta}$ we get
  \begin{align*}
    | Y^{ (n)}_t | & \leq \varepsilon + \Phi_{\varepsilon, \delta} (Y^{(n)}_t)\\
    & = \varepsilon + \int_0^t \bigg( \Phi_{\varepsilon, \delta}' (Y^{(n)}_u)
    (b (X_u) - b (X^{(n)}_{\kappa_n (u)})) + \frac{1}{2} \Phi_{\varepsilon,
    \delta}'' (Y^{(n)}_u) | \sigma (X_u) - \sigma (X^{(n)}_{\kappa_n (u)}) |^2
    \bigg) \dd \langle S \rangle_u\\ 
    &\quad \quad+ M^{(n)}_t .
  \end{align*}
  The contribution from the drift is
  \begin{align*}
    \int_0^t & \Phi_{\varepsilon, \delta}' (Y^{(n)}_u)\big(b (X_u) - b(X^{(n)}_{\kappa_n (u)})\big) \dd \langle S \rangle_u\\
    & = \int_0^t \Phi_{\varepsilon,\delta}' (Y^{(n)}_u)\big (b (X_u) - b (X^{(n)}_u)\big) \dd \langle S\rangle_u
    + \int_0^t \Phi_{\varepsilon, \delta}' (Y^{(n)}_u) \big(b (X^{(n)}_u)
    - b (X^{(n)}_{\kappa_n (u)})\big) \dd \langle S \rangle_u\\
    & \leq C_b \int_0^t | Y^{(n)}_u |\dd \langle S \rangle_u + C_b C_{b,\sigma} 2^{- n} \langle S \rangle_t,
  \end{align*}
  where we used that $| \Phi_{\varepsilon, \delta}' | \leq 1$. The contribution from the quadratic variation is
  \begin{align*}
    \int_0^t \frac{1}{2} &\Phi_{\varepsilon, \delta}'' (Y^{(n)}_u) | \sigma(X_u) - \sigma (X^{(n)}_{\kappa_n (u)}) |^2 \dd \langle S \rangle_u\\
    &\leq \int_0^t \Phi_{\varepsilon, \delta}'' (Y^{(n)}_u) | \sigma (X_u) - \sigma (X^{(n)}_u) |^2 \dd \langle S \rangle_u
    + \int_0^t \Phi_{\varepsilon, \delta}'' (Y^{(n)}_u) | \sigma
    (X^{(n)}_u) - \sigma (X^{(n)}_{\kappa_n (u)}) |^2 \dd \langle S\rangle_u\\
    &\leq C_{\sigma}^2 \int_0^t \Phi_{\varepsilon, \delta}'' (Y^{(n)}_u) |X_u-X^{(n)}_u| \dd \langle S \rangle_u
    +C_{\sigma}^2 \int_0^t \Phi_{\varepsilon, \delta}''(Y^{(n)}_u) |X^{(n)}_u -X^{(n)}_{\kappa_n (u)} | \dd \langle S\rangle_u\\
    & \leq C_{\sigma}^2 \frac{2}{\log (\delta)} \langle S \rangle_t +
    C_{\sigma}^2 \frac{2 \delta}{\varepsilon \log (\delta)}2^{- n} \langle S\rangle_t ,
  \end{align*}
  where we used Assumption~\ref{ass:regularity} in the second last line and the estimate~\eqref{eq:Phi second derivative} is the last one. So overall
  \begin{equation*}
    | Y^{ (n)}_t | 
    \leq \varepsilon + C_b \int_0^t | Y^{(n)}_u | \dd \langle S \rangle_u
    + C_b C_{b,\sigma}2^{- n} \langle S \rangle_t 
    + C_{\sigma}^2 \frac{2}{\log(\delta)}\langle S \rangle_t + C_{\sigma}^2 \frac{2 \delta}{\varepsilon \log (\delta)} 2^{- n} \langle S \rangle_t + M^{(n)}_t . 
  \end{equation*}
  Choosing $\varepsilon = 1 / n$ and $\delta = 2^{n / 2}$ this becomes
  \begin{align*}
    | Y^{ (n)}_t | 
    &\leq \frac{1}{n} + C_b \int_0^t | Y^{(n)}_u | \dd \langle S \rangle_u + C \langle S \rangle_t (2^{- \frac{n}{2}} + n^{- 1}) + M^{(n)}_t \\
    &\leq \frac{1 + C \langle S \rangle_t}{n} + C_b \int_0^t | Y^{(n)}_u | \dd \langle S \rangle_u  + M^{(n)}_t, 
  \end{align*}
  for some $C > 0$ which depends on $C_b, C_{\sigma}$ and $C_{b,\sigma}$ only.
  
  The quadratic variation of $(M^{(n)}_t)_{t\in[0,\infty)}$ is
  \begin{align*}
     \langle M^{(n)} \rangle_t 
    & = \int_0^t | \Phi'_{\varepsilon, \delta}
    (Y^{(n)}_u) (\sigma (X_u) - \sigma (X^{(n)}_{\kappa_n (u)})) |^2 \dd
    \langle S \rangle_u\\
    &\leq   2 C_{\sigma}^2 \int_0^t | Y^{(n)}_u | \dd \langle S \rangle_u + 2
    C_{\sigma}^2 C_{b,\sigma}\int_0^t 2^{- n} \dd \langle S \rangle_u,
  \end{align*}
  for $t\in [0,\infty)$, as claimed.

  \textit{Step~2:} Under the Wiener measure~$\mathbb{W}$ we have $\langle S \rangle_t = t$. Using the estimate~\eqref{eq:Y-n-bound} and a standard localization argument gives 
  \begin{align*}
    \E[| Y^{(n)}_t |]
    \leq \frac{1 + C t}{n} + C_b \int_0^t \E[| Y^{(n)}_u|] \dd u , \quad t\in [0,\infty).
  \end{align*}
  Hence, applying Gronwall's inequality leads to 
  \begin{equation}\label{eq:Y-n-bound 2}
    \E[|Y^{(n)}_t |] \leq \frac{1 + C t}{n}, \quad t\in [0,\infty),
  \end{equation}
  where $C > 0$ depends on $C_b, C_{\sigma}$ and $C_{b,\sigma}$ only, but may differ from the constant denoted by $C$ in Step~1. The Burkholder--Davis--Gundy inequality together with the estimate~\eqref{eq:Y-n-bound} and~\eqref{eq:Y-n-bound 2} yields
  \begin{align*}
    \mathbb{E} \left[\sup_{u \in [0,t]} | Y^{(n)}_u |\right] 
    \leq & \frac{1 + Ct}{n} + C_b \int_0^t \mathbb{E} [| Y^{(n)}_u |] \dd u +
     \mathbb{E} \left[\sup_{u \in [0,t]} | M^{(n)}_u |\right]  \\
    \leq & \frac{1 + Ct}{n} + C_b \int_0^t \mathbb{E} [| Y^{(n)}_u |] \dd u +\bigg(C_{\text{BDG}}
    2 C_{\sigma}^2 C_{b,\sigma} \int_0^t (\mathbb{E}[| Y^{(n)}_u |] + 2^{- n}) \dd u\bigg)^{1/2}\\
    \leq & \frac{1 + C t}{n}+ C \int_0^t \mathbb{E} \left[\sup_{r \in [0,u]} | Y^{(n)}_r |\right] \dd u
    + (C_{\text{BDG}} 2 C_{\sigma}^2((1 + C t)+t))^{1/2} n^{-1/2},
  \end{align*}  
  for a new constant $C > 0$ which depends on $C_b, C_{\sigma}$ and $C_{b,\sigma}$. So the claimed estimate~\eqref{eq:L1 convergence} follows by applying again Gronwall's inequality.
  
  \textit{Step~3:} After having established~\eqref{eq:L1 convergence} the almost sure locally uniform convergence of $(X^{(n^3)}_{\cdot})_{n\in \N}$ to $(X_t)_{t\in [0,\infty)}$ follows by a routine argument: For any $T>0$ we have 
  \[
    \E\left[ \sum_{n=1}^\infty \sup_{t\in [0,T]} |X^{(n^3)}_t- X_t|\right] \le C^3 \sum_{n=1}^\infty n^{-3/2} < \infty,
  \]
  and therefore almost surely $\sum_{n=1}^\infty \sup_{t\in [0,T]} |X^{(n^3)}_t- X_t| < \infty$. Hence, $(X^{(n^3)}_{\cdot})_{n\in \N}$ converges almost surely uniformly on $[0,T]$ to $(X_t)_{t\in [0,\infty)}$. By choosing a countable sequence $T_m \to \infty$, we obtain, almost surely, the locally uniform convergence of $(X^{(n^3)}_{\cdot})_{n\in \N}$ to $(X_t)_{t\in [0,\infty)}$.
\end{proof}

\begin{proof}[Proof of Proposition~\ref{prop:existence result}]
  Consider the event 
  \begin{equation*}
     E_1 := \bigg\{ \omega \in \Omega : X^{(n^3)}(\omega) \text{ converges locally uniformly} \bigg\}
  \end{equation*}
  and denote by $E^c_1$ the complement of the set~$E_1$. We want to show that $E^{c}_1$ is time-superinvariant in the sense of Definition~\ref{def:time-invariant} in order to apply the pathwise Dambis--Dubins--Schwarz theorem (Theorem~\ref{thm:DDS theorem}). For this purpose, it is sufficient to show that $\omega\in E_1$ implies $\omega \circ \phi \in E_1$ for every $\phi\in \mathcal{G}_0$.
  
  Let $\omega \in E_1$ and $\phi\in \mathcal{G}_0$ be a time-change. Thanks to Lemma~\ref{lem:time-change quadratic variation} and the definition of the stopping times $(\tau_k^n)$, for $n,k\in \N$, we have
  \begin{equation*}
    S_{\tau^n_k(\omega \circ \phi)}(\omega \circ \phi ) = S_{\tau^n_k(\omega)}(\omega)
    \quad\text{and}\quad
    \langle S\rangle_{\tau^n_k(\omega \circ \phi)}(\omega \circ \phi ) = \langle S\rangle_{\tau^n_k(\omega)}(\omega).
  \end{equation*}
  Furthermore, we have $X^{(n)}_0(\omega)=x_0=X^{(n)}_0(\omega \circ \phi)$. Suppose now that 
  \begin{equation*}
    X^{(n)}_{\tau^n_k (\omega)}(\omega)=X^{(n)}_{\tau^n_k (\omega \circ \phi)}(\omega \circ \phi)
  \end{equation*}
  for some $k\in \N $ and let us apply an induction argument over $k\in \N$. For $t\in (\tau^n_k (\omega \circ \phi),\tau^n_{k+1} (\omega \circ \phi)]$ we observe that 
  \begin{align*}
    X^{(n)}_t(\omega \circ \phi) 
    &=X^{(n)}_{\tau^n_k (\omega \circ \phi)}(\omega \circ \phi) + b(X^{(n)}_{\tau^n_k (\omega \circ \phi)}(\omega \circ \phi)) 
    \big(\langle S\rangle_t (\omega \circ \phi)-\langle S\rangle_{\tau^n_k (\omega \circ \phi)} (\omega \circ \phi)\big)\\
    & \quad\quad +  \sigma (X^{(n)}_{\tau^n_k (\omega \circ \phi)}(\omega \circ \phi)) \big(S_{t} (\omega \circ \phi) -S_{\tau^n_k (\omega \circ \phi)} (\omega \circ \phi)\big)\\
    &\quad=X^{(n)}_{\tau^n_k (\omega)}(\omega) 
    + b(X^{(n)}_{\tau^n_k (\omega)}(\omega)) \big( \langle S \rangle_{\phi(t)} (\omega)-\langle S\rangle_{\tau^n_k (\omega)} (\omega)  \big)\\
    & \quad\quad +  \sigma (X^{(n)}_{\tau^n_k (\omega )}(\omega)) \big(S_{\phi(t)} (\omega)-S_{\tau^n_k (\omega)} (\omega)\big)\\
    &\quad= (X^{(n)})_{\phi(t)}(\omega).
  \end{align*}
  This implies $X^{(n)}_t(\omega \circ \phi)=X^{(n)}_{\phi(t)}(\omega) $ for all $t\in [0,\infty)$. Hence, if $(X^{(n^3)}_{\cdot}(\omega))_{n\in\N}$ converges locally uniformly, then $(X^{(n^3)}_{\cdot}(\omega \circ \phi))_{n\in \N}$ converges locally uniformly. This means $\omega\in E_1$ implies $\omega\circ\phi\in E_1$, that is $E^c_1$ is time-superinvariant. Combining Theorem~\ref{thm:DDS theorem} and Lemma~\ref{lem:approximation under Wiener measure} ensure that $\overline{P}(E^c_1)=0$. Hence, for typical paths, the process $(X_t)_{t\in [0,\infty)}$, defined by $X_{t}:= \lim_{n\to\infty} X_t^{(n^3)}$, exists as locally uniform limit. Furthermore, by Lemma~\ref{lem:integral} and the continuity property of Riemann--Stieltjes integration, we deduce that $(X_t)_{t\in [0,\infty)}$ is a solution to~\eqref{eq:SDE}, which completes the proof.
\end{proof}

\subsection{Yamada--Watanabe theorem}\label{subsec:YW theorem} 

In probability theory, a famous theorem of Yamada--Watanabe states that there exists a unique solution to one-dimensional stochastic differential equations driven by a Wiener process assuming that the diffusion coefficient is H{\"o}lder continuous of order~$1/2$, see \cite{Yamada1971,Watanabe1971}. This is a very deep insight as usually differential equations require basically Lipschitz continuous coefficients to ensure the uniqueness of solutions. 

\medskip

In this subsection we prove a Yamada--Watanabe type theorem (Theorem~\ref{thm:Yamada--Watanabe}) for the differential equation~\eqref{eq:SDE}, which is driven by typical paths. In order to recall the definition of the set~$\mathrm{DS}$, we refer to Definition~\ref{def:DS}. 

\begin{theorem}\label{thm:Yamada--Watanabe}
  Suppose that Assumption~\ref{ass:regularity} holds. Then, there exists a unique solution $X=(X_t)_{t\in[0,\infty)}$ in $ \mathrm{DS}$ to the differential equation~\eqref{eq:SDE}. 
\end{theorem}

Before we proceed to the proof of Theorem~\ref{thm:Yamada--Watanabe} we will prove some preparatory results. Let $(X_t)_{t\in [0,\infty)}$ be the solution to the differential equation~\eqref{eq:SDE}. From Corollary~\ref{cor:convergence of integrals} we know that if the stopping times $(\rho_k^n)$ are given by $\rho_{0}^{n}:=0$ and 
\begin{equation*}
  \rho_{k}^{n}:=\inf\{t>\rho_{k-1}^{n}\,:\,|\sigma(X_{t})-\sigma(X_{\rho_{k-1}^{n}})|\geq2^{-n}\text{ or }|S_{t}-S_{\rho_{k-1}^{n}}|\geq2^{-n}\}
\end{equation*}
for $n,k\in\mathbb{N}$, then for typical paths 
\begin{equation}\label{eq:ito integral appoximation}
  \int_{0}^{t}\sigma(X_{u})\dd S_{u}:=\lim_{n\to\infty}\sum_{k=0}^{\infty}\sigma(X_{\rho_{k}^{n}})\big(S_{\rho_{k+1}^{n}\wedge t}-S_{\rho_{k}^{n}\wedge t}\big),\quad t\in[0,\infty),
\end{equation}
and the convergence in~\eqref{eq:ito integral appoximation} is locally uniform, for typical paths, since the stopping times $(\rho_k^n)$ ensure that the integrand $\sigma (X_u)$ is uniformly approximated. From this approximation one can also derive the behaviour of the model-free It{\^o} integral under time-changes. 

\begin{lemma}\label{lem:time-change integral} 
  Let $\phi\colon[0,\infty)\to[0,\infty)$ be a time-change in $\mathcal{G}_0$ and $\omega \in \Omega$. If the limit
  \begin{equation*}
    \int_{0}^{t}\sigma(X_{u})\dd S_{u} (\omega):=\lim_{n\to\infty}\sum_{k=0}^{\infty}\sigma(X_{\rho_{k}^{n}((\omega))}(\omega))\big(S_{\rho_{k+1}^{n}(\omega)\wedge t}(\omega)-S_{\rho_{k}^{n}(\omega)\wedge t}(\omega)\big),\quad t\in[0,\infty),
  \end{equation*}
  exists as locally uniform limit, then 
  \begin{equation*}
    \int_{0}^{t}\sigma((X\circ\phi)_{u})\dd(S\circ\phi)_{u}(\omega), \quad t\in [0,\infty),  
  \end{equation*}
  as defined in \eqref{eq:ito integral appoximation}, exists as locally uniform limit and 
  \begin{equation*}
    \int_{0}^{\phi(t)}\sigma(X_{u})\dd S_{u} (\omega)=\int_{0}^{t}\sigma((X\circ\phi)_{u})\dd(S\circ\phi)_{u}(\omega).
  \end{equation*}
  Furthermore, if $\omega\in \Omega$ is such that the quadratic variation $(\langle S\rangle_t (\omega))_{t\in [0,\infty)}$ exists, then 
  \begin{equation*}
    \int_0^{\phi(t)} b(X_u(\omega)) \dd \langle S\rangle_u (\omega)=
    \int_0^{t} b(X_u(\omega\circ \phi)) \dd \langle S\rangle_u (\omega\circ \phi),\quad t\in [0,\infty).
  \end{equation*}
\end{lemma}

\begin{proof}
  By the definition of the stopping times~$(\rho_k^n)$, for $n\in\mathbb{N}$ and $t\in [0,\infty)$ we observe that 
  \begin{align*} 
    \int_{0}^{\phi(t)}\sigma(X_{s})\dd S_{s} (\omega)
    &:= \lim_{n\to\infty} \sum_{k=0}^{\infty}\sigma(X_{\rho_{k}^{n}}(\omega))\big(S_{\rho_{k+1}^{n}\wedge\phi(t)}(\omega)-S_{\rho_{k}^{n}\wedge\phi(t)}(\omega)\big) \\
    &= \lim_{n\to\infty} \sum_{k=0}^{\infty}\sigma(X_{\tilde{\tau}_{k}^{n}}(\omega))\big(S_{\tilde{\tau}_{k+1}^{n}\wedge t}(\omega)-S_{\tilde{\tau}_{k}^{n}\wedge t}(\omega)\big) 
  \end{align*}
  where the last equality holds for the new stopping times $(\tilde{\tau}_k^n)$  defined by $\tilde{\tau}_{0}^{n}:=0$ and 
  \begin{equation*}
    \tilde{\tau}_{k}^{n}:=\inf\{t>\tilde{\tau}_{k-1}^{n}\,:\,|\sigma((X\circ\phi)_{t})-\sigma((X\circ\phi)_{\tilde{\tau}_{k-1}^{n}})|\geq 2^{-n}\text{ or }|(S\circ\phi)_t - (S\circ\phi)_{\tilde{\tau}_{k-1}^{n}\wedge t}|\geq2^{-n}\}
  \end{equation*}
  for $n,k\in\mathbb{N}$. Hence, 
  \begin{align*}
    \int_{0}^{t}\sigma((X\circ\phi)_{s})\dd(S\circ\phi)_{s}(\omega)
    &   = \lim_{n\to\infty}  \sum_{k=0}^{\infty}\sigma(X_{\tilde{\tau}_{k}^{n}}(\omega))\big(S_{\tilde{\tau}_{k+1}^{n}\wedge t}(\omega)-S_{\tilde{\tau}_{k}^{n}\wedge t}(\omega)\big) \\
    &   =   \int_{0}^{\phi(t)}\sigma(X_{s})\dd S_{s} (\omega),
  \end{align*}
  which reveals the first assertion. 
  
  The second assertion follows by Lemma~\ref{lem:time-change quadratic variation}
\end{proof}

With this preparatory results at hand we are in a position to prove Theorem~\ref{thm:Yamada--Watanabe}.

\begin{proof}[Proof of Theorem~\ref{thm:Yamada--Watanabe}]
  Since the existence of a solution $(X_t)_{t\in [0,\infty)}$ to~\eqref{eq:SDE} in $\mathrm{DS}$ follows by Proposition~\ref{prop:existence result}, it remains to show uniqueness in $\mathrm{DS}$. 
  
  Let us suppose that there are two continuous and adapted processes $X^{(1)}=(X^{(1)}_t)_{t\in [0,\infty)}$ and $X^{(2)}=(X^{(2)}_t)_{t\in [0,\infty)}$ solving the differential equation~\eqref{eq:SDE} driven by typical paths. Let us consider the event 
  \begin{equation*}
    E_{2}:=\bigg\{\omega\in \mathrm{DS}\;:\; \sup_{t\in [0,\infty)}\|X_t^{(1)}(\omega)-X_t^{(2)}(\omega)\|>0\bigg\}.
  \end{equation*}
  We shall show that the event $E_{2}$ is time-superinvariant in the sense of Definition~\ref{def:time-invariant}. Let $\phi\in\mathcal{G}$ and $\omega\in \mathrm{DS}$. Without loss of generality we may assume that for $\omega$ the quadratic variation $(\langle S\rangle_t (\omega))_{t\in [0,\infty)}$ in the sense of Proposition~\ref{prop:quadratic variation} and $(X^{(1)}_t(\omega))_{t\in [0,\infty)}$ and $(X_t^{(2)}(\omega))_{t\in[0,\infty)}$ satisfy equation~\eqref{eq:SDE} together with~\eqref{eq:ito integral appoximation}. Due to Lemma~\ref{lem:time-change integral}, for $i=1,2$ we observe that 
  \begin{align*}
    X_t^{(i)}(\omega\circ \phi) 
    &=  x_{0}
    + \int_0^t b(X_u^{(i)}(\omega\circ \phi))\dd \langle S \rangle_u (\omega\circ \phi)
    +\int_{0}^{t}\sigma(X_{u}^{(i)})\dd S_{u} (\omega\circ \phi)\\
    &=  x_{0}
    + \int_0^{\phi(t)} b(X^{(i)}_u(\omega))\dd \langle S \rangle_u (\omega)
    +\int_{0}^{\phi(t)}\sigma(X^{(i)}_{u})\dd S_{u} (\omega)
    = X^{(i)}_{\phi(t)}(\omega)
  \end{align*}
  for $t\in [0,\infty)$. This reveals that $\omega\circ\phi\in E_{2}$ implies $\omega \in E_{2}$. Hence, $E_{2}$ is time-invariant and by Lemma~\ref{lem:timeinvariant} $E_{2}$ is also time-superinvariant. By Vovk's pathwise Dambis--Dubins--Schwarz theorem (Theorem~\ref{thm:DDS theorem} and \cite[Corollary~3.7]{Vovk2012}) and the classical uniqueness result of Yamada and Watanabe for stochastic differential equations driven by a Brownian motion (\cite[Chapter~5.2, Proposition~2.13]{Karatzas1988}), we obtain $\overline{P}(E_{2};\mathrm{DS})=\mathbb{W}(E_{2})=0$.
\end{proof}

\begin{remark}\label{rem:no boundedness}
  The boundedness assumption on $b$ and $\sigma$ made in Assumption~\ref{ass:regularity} is not needed for the uniqueness result provided in Theorem~\ref{thm:Yamada--Watanabe}. However, it is required by Proposition~\ref{prop:existence result} to provide the existence of a solution. 
\end{remark}

\subsection{Doss--Sussmann approximation}\label{subsec:DS theorem}

While differential equations of the form~\eqref{eq:SDE} appear frequently in the mathematical modeling of random phenomena, the actual ``noise'' term~$(S_t)_{t\in [0,\infty)}$ in many applications is represented by a process of bounded variation. This led to the natural question how ordinary differential equations, perturbed by random noises of bounded variation, and stochastic differential equations are linked. A first answer to this fundamental question was given by Doss~\cite{Doss1977} and Sussmann~\cite{Sussmann1978}, stating that a suitably chosen sequence of ordinary differential equations can approximate a differential equation driven by typical paths.

\medskip

For $n\in \N$ let $(S^{(n)}_t)_{t\in [0,\infty)}$ be a process such that function $t\mapsto S^{(n)}_t(\omega)$ is continuous and of locally bounded variation for every $\omega\in \Omega$. Let us consider the differential equation
\begin{equation}\label{eq:SDE approximation}
  X_t^{(n)} = x_0 + \int_0^t b(X_u^{(n)})\dd \langle S\rangle_u + \int_0^t \sigma(X_u^{(n)})\dd S_u^{(n)},\quad t\in [0,\infty),
\end{equation}
where $x_0\in \R$, $b\colon\R\to \R $ and $\sigma\colon\R\to \R$ are Lipschitz functions and $X^{(n)}\colon \Omega \times [0,\infty)\to\R$. Notice that the differential equation~\eqref{eq:SDE approximation} possesses a unique solution $(X^{(n)}_t(\omega))_{t\in [0,\infty)}$ for all $\omega \in \Omega$ such that $(\langle S\rangle_t(\omega))_{t\in [0,\infty)}$ exists in the sense of Proposition~\ref{prop:quadratic variation}, see \cite[Corollary~3.9]{Friz2010}. Under suitable assumptions, if the sequence $(S^{(n)}_{\cdot})_{n\in\N}$ uniformly approximates the coordinate process $(S_t)_{t\in [0,\infty)}$, it turns out that the corresponding sequence $(X^{(n)}_{\cdot})_{n\in \N}$ of solutions indeed converges to an adapted and continuous process~$(X_t)_{t\in [0,\infty)}$ solving a differential equation driven by typical paths.

\begin{remark}
  A natural and most straightforward way to approximate uniformly the coordinate process $(S_t)_{t\in [0,\infty)}$ with continuous processes $(S^{(n)}_{\cdot})_{n\in \N}$ of locally bounded variation is to choose $(S^{(n)}_{\cdot})_{n\in \N}$ as a piecewise linear approximation along the sequence of Lebesgue stopping times, cf. \eqref{eq:Lebesgue stopping times}. While these approximations may be not adapted, one can use the concept of truncated variation, as in \cite{Lochowski2013}, to obtain uniform, continuous and adapted approximations of $(S_t)_{t\in [0,\infty)}$, which have locally bounded variation.
\end{remark}

The following theorem formulates precisely the indicated convergence result for typical paths and can be seen as a game-theoretic version of the Doss--Sussmann approximation result in probability theory.

\begin{theorem}
  Suppose that $b\colon \R \to\R$ is Lipschitz continuous and $\sigma \colon \R\to \R$ is twice continuously differentiable with bounded first and second derivative. Let $(S^{(n)}_{\cdot})_{n\in\N}$ be a sequence of processes such that the function $t\mapsto S^{(n)}_t(\omega)$ is continuous and of locally bounded variation for every $\omega\in \Omega$, and denote by $(X^{(n)}_{\cdot})_{n\in \N}$ the sequence of corresponding solutions to~\eqref{eq:SDE approximation}.
  
  If $\lim_{n\to \infty}\|S^{(n)}-S\|_{\infty;[0,T]}=0$, then
  \begin{equation*}
    \lim_{n\to \infty}\|X^{(n)}-X\|_{\infty;[0,T]}=0, \quad \text{for typical paths},
  \end{equation*}
  where $X=(X_t)_{t\in [0,\infty)}$ denotes the solution to 
  \begin{equation}\label{eq:SDE limit}
     X_t = x_0 + \int_0^t \bigg[b(X_u)+\frac{1}{2}\sigma (X_u)\sigma^\prime(X_u)\bigg]\dd \langle S\rangle_u + \int_0^t \sigma(X_u)\dd S_u,\quad t\in [0,\infty),
  \end{equation}
  for typical paths, which is unique for typical paths in $\mathrm{DS}$.
\end{theorem}

The proof adapts the classical probabilistic arguments, cf. e.g. \cite[Section~5.2~D]{Karatzas1988}.

\begin{proof}
  Without loss of generality, we consider $\omega \in \Omega$ such that $(\langle S\rangle_t (\omega))_{t\in [0,\infty)}$ exists in the sense of Proposition~\ref{prop:quadratic variation}, recalling that the complement of this set has outer measure zero.

  \textit{Step~1:} Let $g\colon \R^2 \to \R$ be the solution to the ordinary differential equation
  \begin{equation*}
    \frac{\partial g }{\partial x} = \sigma (g)\quad \text{and} \quad
    g(0,y) = y,\quad y\in \R.
  \end{equation*}
  Hence, for $x,y\in \R$ we get
  \begin{equation*}
    \frac{\partial^2 g}{\partial x^2} = \sigma (g)\sigma^\prime (g),\quad 
    \frac{\partial^2 g}{\partial x \partial y} = \sigma^\prime (g)\frac{\partial g}{\partial y}\quad\text{and}\quad
    \frac{\partial}{\partial y} g(0,y)=1,
  \end{equation*}
  which gives that 
  \begin{equation*}
    \frac{1}{\rho(x,y)}:=\frac{\partial}{\partial y} g(x,y) = \exp \left\{ \int_0^x  \sigma'( g(z,y)) \dd z \right\} >0.  
  \end{equation*}
  As shown in the proof of \cite[Chapter~5, Proposition~2.21]{Karatzas1988}, the continuous function
  \begin{equation*}
    f(x,y):= \rho (x,y) b(g(x,y)),\quad x,y\in \R,
  \end{equation*}
  is locally Lipschitz continuous in~$y$, bounded in $x$ and has locally linear growth in~$y$. Hence, due to \cite[Corollary~3.9]{Friz2010}, there exists a unique solution $(Y_{t}(\omega))_{t\in [0,\infty)}$ to the ordinary differential equation
  \begin{equation}\label{eq:equation Y}
    Y_t(\omega) = x_0 + \int_0^t f(S_u(\omega),Y_u(\omega)) \dd \langle S\rangle_u (\omega),
    \quad t\in [0,\infty),
  \end{equation}
  since the coefficient~$f$ is a locally Lipschitz continuous function of linear growth. This allows us to define the process $X_t := g(S_t,Y_t)$ for $t\in [0,\infty)$. Using F{\"o}llmer's pathwise It{\^o} formula~\cite{Follmer1981}, we see that 
  \begin{align*}
    X_t (\omega)
    &= g(S_0(\omega),Y_0(\omega))+ \int_0^t \frac{\partial }{\partial x} g(S_u(\omega),Y_u(\omega)) \dd S_u(\omega)\\
    &\qquad+ \frac{1}{2} \int_0^t \frac{\partial^2}{\partial x^2} g(S_u(\omega),Y_u(\omega)) \dd \langle S\rangle_u(\omega)
    +\int_0^t b(X_u(\omega)) \dd \langle S\rangle_u(\omega)\\
    &= x_0+ \int_0^t \sigma (X_u(\omega)) \dd S_u(\omega)
    +\int_0^t \bigg[ \frac{1}{2} \sigma (X_u(\omega))\sigma^{\prime}(X_u(\omega)) + b(X_u(\omega))\bigg] \dd \langle S\rangle_u(\omega),
  \end{align*}
  for $t\in [0,\infty)$. Hence, $(X_t(\omega))_{t\in [0,\infty)}$ is the unique solution to~\eqref{eq:SDE limit}. Note that the uniqueness holds due to Remark~\ref{rem:no boundedness}.
  
  \textit{Step~2:} Similar to Step~1, since $(S^{(n)}_t(\omega))_{t\in [0,\infty)}$ is of locally bounded variation, there exists a unique solution $(Y^{(n)}_t(\omega))_{t\in [0,\infty)}$ to the differential equation 
  \begin{equation}\label{eq:equation Yn}
    Y^{(n)}_t(\omega) = x_0 + \int_0^t f(S_u^{(n)}(\omega),Y_u^{(n)}(\omega)) \dd \langle S \rangle_u(\omega),\quad t\in [0,\infty).
  \end{equation}
  Moreover, we get that $X^{(n)}_{\cdot}:=g(S^{(n)}_{\cdot},Y^{(n)}_{\cdot})$ is the unique solution to~\eqref{eq:SDE approximation}. Indeed, by classical calculus, for $t\in [0,\infty)$, we have 
  \begin{align*}
    X^{(n)}_t (\omega)
    &= x_0+ \int_0^t \frac{\partial }{\partial x} g(S^{(n)}_u(\omega),Y^{(n)}_u(\omega)) \dd S^{(n)}_u(\omega)
    + \int_0^t \frac{\partial }{\partial y} g (S^{(n)}_u(\omega),Y^{(n)}_u(\omega))\dd Y^{(n)}_u(\omega)\\
    &= x_0+ \int_0^t \sigma (X^{(n)}_u(\omega)) \dd S^{(n)}_u(\omega)
    + \int_0^t b(X^{(n)}_u(\omega))\dd \langle S\rangle_u(\omega).
  \end{align*}
   
  \textit{Step~3:} Recall that $X_{\cdot}=g(S_{\cdot},Y_{\cdot})$ and $X^{(n)}_{\cdot}=g(S^{(n)}_{\cdot},Y^{(n)}_{\cdot})$. Hence, since $g$ is a continuous function and $(S^{(n)}_{\cdot}(\omega))_{n\in \N}$ converges locally uniformly to $(S_t(\omega))_{t\in [0,\infty)}$, it is sufficient to prove that $(Y^{(n)}_{\cdot}(\omega))_{t\in [0,\infty)}$ converges locally uniformly to $(Y_{t}(\omega))_{t\in [0,\infty)}$ in order to show that $(X^{(n)}_\cdot(\omega))_{n\in \N}$ converges locally uniformly to $(X_t(\omega))_{t\in [0,\infty)}$. Furthermore, notice that $(Y^{(n)}_t(\omega))_{t\in [0,\infty)}$ and $(Y_t(\omega))_{t\in [0,\infty)}$ are solutions to the ordinary differential equations~\eqref{eq:equation Yn} and~\eqref{eq:equation Y}, respectively, which are both driven by processes of locally bounded variation. Therefore, the locally uniform convergence of $(S^{(n)}_{\cdot}(\omega))_{n\in \N}$  to $(S_t(\omega))_{t\in [0,\infty)}$ implies the locally uniform convergence of $(Y^{(n)}_{\cdot}(\omega))_{t\in [0,\infty)}$ to $(Y_{t}(\omega))_{t\in [0,\infty)}$, which is a classical stability result for ordinary differential equations, see for instance \cite[Theorem~3.15]{Friz2010}.
\end{proof}

\begin{remark}
  If the drift term is of the form $b(x) = \sigma(x) \tilde b(x)$ for a locally Lipschitz continuous function $\tilde{b}$ of linear growth, then it suffices to assume that $\sigma$ is continuously differentiable and of linear growth, and that at least one of the two functions $\tilde b$ or $\sigma$ is bounded. Indeed, in that case we can apply the Lamperti transform rather than the Doss--Sussmann transform: Let $g$ be as in the previous proof. Note that $\frac{\partial}{\partial x} g(x, y) = \sigma(g(x,y))$, and since $\sigma$ is continuously differentiable, we get $\frac{\partial^2}{\partial x^2} g(x,y) = \sigma'(g(x,y)) \sigma(g(x,y))$, so $g(\cdot, y)$ is twice continuously differentiable even if $\sigma$ is only continuously differentiable. In particular, the map $\tilde b(g(\cdot,x_0))$ is locally Lipschitz continuous as concatenation of locally Lipschitz continuous maps. If $\sigma$ is bounded, then $\tilde b(g(\cdot,x_0))$ is of linear growth as concatenation of maps of linear growth, while if $\tilde b$ is bounded, of course also $\tilde b(g(\cdot,x_0))$ is bounded. Therefore, it follows again from \cite[Corollary~3.9]{Friz2010} that, for typical path $\omega \in \Omega$, there is a unique pathwise solution $(Z_t(\omega))_{t\in [0,\infty)}$ to the equation
  \[
    Z_t(\omega) = \int_0^t \tilde b(g(Z_s(\omega)+S_s(\omega),x_0))\dd \langle S\rangle_s(\omega),\qquad t \in [0,\infty).
  \]
  Then, we obtain from F\"ollmer's pathwise It\^o formula that $(X_t(\omega))_{t\in [0,\infty)}$, given by $X_t(\omega):= g(Z_t(\omega) + S_t(\omega),x_0)$, solves
  \begin{align*}
    X_t(\omega) =\ & g(0,x_0) + \int_0^t \sigma(u(Z_s(\omega) + S_s(\omega),x_0))\tilde b(g(Z_s(\omega)+S_s(\omega),x_0))\dd \langle S\rangle_s(\omega) \\
    & + \int_0^t \sigma(g(Z_s(\omega) + S_s(\omega),x_0))\dd S_s(\omega)  \\
    & + \frac12 \int_0^t \sigma'(g(Z_s(\omega) + S_s(\omega),x_0))\sigma(g(Z_s(\omega) + S_s(\omega),x_0)) \dd \langle S\rangle_s(\omega) \\
    =\ & x_0 + \int_0^t \left[ \sigma(X_s(\omega))\tilde b(X_s(\omega)) + \frac12 \sigma'(X_s(\omega))\sigma(X_s(\omega))\right] \dd \langle S\rangle_s + \int_0^t \sigma(X_s(\omega))\dd S_s(\omega).
  \end{align*}
  Uniqueness of solutions and the approximation via the solutions $(X^{(n)}_{\cdot})_{n\in \N}$ of~\eqref{eq:SDE approximation} follow exactly as in the previous proof.
  
  See also the recent paper~\cite{Karatzas2016} for closely related results on pathwise solutions of one-dimensional SDEs via Doss--Sussmann and Lamperti transforms in a semi-martingale context.
\end{remark}


\newcommand{\etalchar}[1]{$^{#1}$}
\providecommand{\href}[2]{#2}

\end{document}